\documentclass[12pt,a4paper]{article}
\usepackage[T1]{fontenc}

\usepackage{amsmath,amsthm,amsfonts,amssymb}
\usepackage{enumerate}
\usepackage[bookmarks=false,colorlinks=true,linkcolor=black,citecolor=black,filecolor=black,urlcolor=black]{hyperref}
\usepackage{cleveref}
\usepackage{multirow}
\usepackage{booktabs}
\usepackage[table]{xcolor}

\newtheorem{theorem}{Theorem}[section]
\newtheorem{definition}[theorem]{Definition}

\newtheorem{proposition}[theorem]{Proposition}
\newtheorem{lemma}[theorem]{Lemma}

\newtheorem{remark}[theorem]{Remark}
\newtheorem{conjecture}[theorem]{Conjecture}

\newcommand{\Z}{{\mathbb{Z}}}
\newcommand{\N}{{\mathbb{N}}}
\newcommand{\F}{\mathbb{F}}

\newcommand{\zero}{{\mathbf{0}}}

\newcommand{\ba}{a}
\newcommand{\bb}{b}
\newcommand{\bd}{d}

\newcommand{\bu}{\mathbf{u}}

\newcommand{\be}{\mathbf{e}}

\newcommand{\bomega}{\boldsymbol\omega}

\newcommand{\wt}{\operatorname{wt}}
\newcommand{\Aut}{\operatorname{Aut}}

\newcommand{\Supp}{\operatorname{Supp}}

\newcommand{\Gcd}{\operatorname{gcd}}

\newcommand{\HFP}{\operatorname{HFP}}

\newcommand{\modu}{\operatorname{mod}}
\usepackage{url}
\urldef{\mailsa}\path|{ivan.bailera, joaquim.borges, josep.rifa}@uab.cat|

\title{Hadamard full propelinear codes with associated group $C_{2t}\times C_2$; rank and kernel\footnote{Submitted to \emph{Advances in Mathematics of Communications}. This work has been partially supported by the Spanish grants TIN2016-77918-P (AEI/FEDER, UE).}
}
\author{Ivan Bailera,  Joaquim Borges and Josep Rif\`{a}\\
\small Department of Information and Communications Engineering,\\
\small Universitat Aut\`onoma de Barcelona, Spain\\
\small \mailsa}

\begin{document}

\maketitle
\begin{abstract}
We introduce the Hadamard full propelinear codes that factorize as direct product of groups such that their associated group is $C_{2t}\times C_2$. We study the rank, the dimension of the kernel, and the structure of these codes. For several specific parameters we establish some links from circulant Hadamard matrices and the nonexistence of the codes we study.
We prove that the dimension of the kernel of these codes is bounded by $3$ if the code is nonlinear. We also get an equivalence between circulant complex Hadamard matrix and a type of Hadamard full propelinear code, and we find a new example of circulant complex Hadamard matrix of order $16$.\\

\end{abstract}

\section{Introduction}

The Hadamard conjecture proposes that an Hadamard matrix of order $4t$ exists for every positive integer $t$. The search for a proof of the Hadamard conjecture has stimulated several advances in the fields of design theory and combinatorics. Concepts such as Hadamard groups, difference sets, cocyclic Hadamard matrices and Hadamard full propelinear codes are related in different ways.

Ito \cite{ito}, \cite{itoII}, \cite{itoIII} introduced the concept of Hadamard groups, he showed a relation between Hadamard difference sets and Hadamard groups and he conjectured that the dicyclic group $Q_{8t}$ is always an Hadamard group. Schmidt \cite{sch} verified Ito's conjecture for $1 \leq t \leq 46$. Rif\`a and Su\'arez \cite{rs} showed that the concept of Hadamard group is equivalent to Hadamard full propelinear code. Flannery \cite{flan} proved that the concepts of cocyclic Hadamard matrix and Hadamard group are equivalent. De Launey, Flannery and Horadam \cite{lfh} proved that the existence of a cocyclic Hadamard matrix of order $4t$ is equivalent to the existence of a normal relative difference set with parameters $(4t,2,4t,2t)$. The above equivalences give us different paths to approach to the proof of Ito's conjecture. For example, the cocyclic Hadamard conjecture (de Launey and Horadam) states that there is a cocyclic Hadamard matrix of order $4t$ for all $t \in \N$, that is equivalent to Ito's conjecture, but not conversely. The cocycles over the groups $C_t \times C_2^2$, for $t$ odd, were studied by Baliga and Horadam \cite{hb}. The solution set includes all Williamson Hadamard matrices, so this set of groups is potentially a uniform source for generation of Hadamard matrices. Any Hadamard matrix of order less than or equal to $20$ is cocyclic. For orders less than or equal to $200$, only order $4t=188=4\cdot 47$ is not yet known to have a cocyclic construction.

Let $H$ be a cocyclic Hadamard matrix over a finite group $G$. We call \emph{associated group} to this group $G$ in the context of Hadamard full propelinear codes. An Hadamard full propelinear code, as a group, is a central group extension of $G$ by $C_2$.
In order to study Hadamard full propelinear codes with a determined associated group, it seems that the easiest way to start is to set the cyclic group as associated group. If the associated group is the cyclic group $C_{4t}$ of order $4t$ then the corresponding Hadamard full propelinear code has one of the two following group structures, $C_{8t}$ or $C_{4t} \times C_2$. Rif\`a and Su\'arez \cite{rs} proved that an Hadamard full propelinear code cannot be a cyclic group and Rif\`a \cite{rifacirculant} showed that an Hadamard full propelinear code of type $C_{4t} \times C_2$ is equivalent to a circulant Hadamard code. Ryser's conjecture \cite{rys} says that there do not exist circulant Hadamard matrices for orders greater than $4$. Therefore, the next natural step is to study Hadamard full propelinear codes with associated group whose structure is a product of two cyclic groups, $C_{2t} \times C_2$, which are equivalent to the cocyclic Hadamard matrices over $C_{2t} \times C_2$. Baliga and Horadam \cite{hb} studied this class of codes for the case $t$ odd.

The goal of this paper is to study the rank and dimension of the kernel of the Hadamard full propelinear codes whose group structure consists of a nontrivial direct product of groups such that their associated group is $C_{2t}\times C_2$.
In \Cref{prelim}, we present the preliminaries about Hadamard full propelinear codes and known lemmas about the rank and the dimension of the kernel.
In \Cref{C2tC2}, we introduce the Hadamard full propelinear codes with associated group $C_{2t}\times C_2$ obtaining four kinds of codes that are studied in Subsections \ref{4tu2}, \ref{2t22u}, \ref{2t4u}, and \ref{tQu}.
In each subsection, we study the rank and the dimension of the kernel of every kind of code that we have introduced. 
Also, we prove that the parameter $t$ is even if the Hadamard full propelinear code has an abelian group structure, and the parameter $t$ is odd if the code has a non-abelian group structure. 
In Subsection \ref{4tu2} and Subsection \ref{tQu}, we prove that the dimension of the kernel is equal to one.
In Subsection \ref{2t22u}, we prove that if the dimension of the kernel is greater than one, then either there exists a circulant Hadamard matrix or there exist a code of Subsection \ref{4tu2}. 
In Subsection \ref{2t4u}, we present a conjecture about the nonexistence of a kind of Hadamard full propelinear code that is equivalent to Arasu et al. conjecture \cite{Arasu2002} about the nonexistence of circulant complex Hadamard matrices.
Also, we show that the circulant complex Hadamard matrix of order $2t=16$, introduced by Arasu et al. in \cite{Arasu2002}, corresponds to an Hadamard full propelinear code with a group structure isomorphic to $C_{2t}\times C_4$ with rank equal to $11$ and dimension of the kernel equal to $2$. Moreover, we have found another nonequivalent Hadamard full propelinear code with a group structure isomorphic to $C_{2t}\times C_4$ of length $2t=16$ with rank equal to $13$ and dimension of the kernel equal to $1$ which corresponds to a circulant complex Hadamard matrix.

\section{Preliminaries} \label{prelim}

Let $\F$ be the binary field. The \textit{Hamming distance} between two vectors $x, y \in \F^n$, denoted by $d_H (x, y)$, is the number of the coordinates in which $x$ and $y$ differ.
The \textit{Hamming weight} of $x$ is given by $\wt_H(x)=d_H(x,\be)$, where $\be$ is the all-zero vector.
We denote the all-one vector of length $n$ by $\bu_n$, and the vector $(1,0,\ldots,1,0)$ of length $2n$ by $\bomega_{2n}$, but we will write $\bu$ and $\bomega$ when the length is clear from the context. We denote the complement of $x\in \F^n$, $x+\bu$, by $\bar{x}$.
A $(n,M,d)$-code is a subset, $C$, of $\F^n$ where $\left|C\right|=M$ and $d$ is the greatest value such that $d_H(x,y)\geq d$ for all $x,y\in C$ with $x \neq y$.
The elements of a code are called \textit{codewords} and $d$ is called \textit{minimum distance}.
The parameter $d$ determines the error-correcting capability of $C$ which is given by $e = \lfloor \frac{d-1}{2} \rfloor$.
For a vector $x$ in $\F^n$, the support of $x$, denoted by $\Supp(x)$, is defined as the set of its nonzero positions.
The \textit{rank} of a binary code $C$, $r=rank(C)$, is the dimension of the linear span of $C$.
The \textit{kernel} of a binary code is the set of words which keeps the code invariant by translation, $K(C) := \{z \in \F^n : C + z = C\}$.
The kernel of a code $C$ is a linear subspace and, assuming the all-zero vector is in $C$, we have that $K(C) \subseteq C$. 
We denote the dimension of the kernel of $C$ by $k=ker(C)$.
From now, we assume $\be \in C$ for every code $C$.

\begin{definition}
An Hadamard matrix is a $n \times n$ matrix $H$ containing entries from the set $\{1,-1\}$, with the property that:
\vspace{-0.3cm}

$$HH^T=nI_n,$$

where $I_n$ is the identity matrix of order $n$.
\end{definition}

If $n>2$ it is easy to prove that any three rows (columns) agree in precisely $n/4$ coordinates. Thus, if there is an Hadamard matrix of order $n$, with $n>2$, then $n$ is multiple of 4. We will say $n=4t$. Two Hadamard matrices are equivalent if one can be obtained from the other by permuting rows and/or columns and multiplying them by $-1$. We can change the first row and column of an Hadamard matrix into $+1$'s and we obtain an equivalent Hadamard matrix which is called \textit{normalized}. The matrix obtained from a normalized Hadamard matrix, by replacing all 1's by 0's and all $-1$'s by  1's, is called \textit{binary normalized Hadamard matrix}. The binary code consisting of the rows of a binary Hadamard matrix and their complements is called a (binary) \textit{Hadamard code}, which is of length $n$, with $2n$ codewords, and minimum distance $n/2$.

Let $S_n$ be the symmetric group of permutations of the set $\{1,\ldots,n\}$. For any $\pi \in S_n$ and $x\in \F^n$, $x=(x_1,\ldots,x_n)$, we write $\pi(x)$ to denote $(x_{\pi^{-1}(1)},\ldots,x_{\pi^{-1}(n)})$.
Two binary codes $C_1,C_2$ of length $n$ are said to be \textit{isomorphic} if there is a coordinate permutation $\pi \in S_n$ such that $C_2=\{\pi(x) : x\in C_1\}$. They are said to be \textit{equivalent} if there is a vector $y\in \F^n$ and a coordinate permutation $\pi\in S_n$ such that $C_2=\{y+\pi(x) : x\in C_1\}$.

\begin{definition}\label{hg}\cite{ito}
An Hadamard group $G$ of order $8t$ is a group containing a $4t$-subset $D$ and a central involution $u$ ($D$ is called Hadamard subset corresponding to $u$), such that
\begin{enumerate}[i)]
\item $D$ and $Da$ intersect exactly in $2t$ elements, for any $a \notin \langle u \rangle \subset G$,
\item $Da$ and $\{b,bu\}$ intersect exactly in one element, for any $a,b\in G$.
\end{enumerate}
\end{definition}
We note that $G=D \cup Du$, and $D$ and $Du$ are disjoint. It is immediate from the second condition of the previous definition. 
\begin{lemma}\cite{ito}\label{ordersquare}
Let $G$ be an Hadamard group of order $8t$ such that $G=N \times \langle u \rangle$, where $N$ is a normal subgroup of $G$ of index $2$. Then the order of $N$ is a square.
\end{lemma}
\begin{definition}\cite{but}
A relative $(v,m,k,\lambda)$-difference set in a group $G$ relative to a normal subgroup $N$, where $|G|=vm$ and $|N|=m$, is a subset $D$ of $G$ such that $|D|=k$ and the multiset of quotients $d_1d_2^{-1}$ of distinct elements $d_1,d_2 \in D$ contains each element of $G\setminus N$ exactly $\lambda$ times, and contains no elements of $N$.
\end{definition}
An \textit{Hadamard difference set} in a group $G$ is a relative $(4t,2,4t,2t)$-difference set in a group $G$ of order $8t$. Let $D$ an Hadamard difference set in a group $G$ of order $8t$ relative to a normal subgroup $N \simeq \F$ of $G$. Thus $G$ is an Hadamard group of order $8t$.
The exponent of a group $G$, $\exp (G)$, is the smallest positive integer $n$ such that $g^n = 1$ for all $g \in G$.
\begin{proposition}\cite{kr}\label{exponentbound} 
There exists an Hadamard difference set in an abelian group $G$ of order $2^{2s+2}$ if and only if $\exp(G) \leq 2^{s+2}$.
\end{proposition}
\begin{proposition}\label{quotientbound}\cite{dil}
Let $C$ be an Hadamard group of order $2^{2s+2}$ and suppose that $C$ has a normal subgroup $G$ such that $C/G$ is cyclic. Then $C/G$ has order at most $2^{s+2}$.
\end{proposition}

\begin{definition}\cite{rbh}
A binary code $C$ of length n has a propelinear structure if for each codeword $x \in C$ there exists $\pi_x\in S_n$ satisfying the following conditions for all $y \in C$:
\begin{itemize}
\item[(i)] $x+\pi_x(y) \in C$,
\item[(ii)] $\pi_x\pi_y=\pi_z$, where $z=x+\pi_x(y)$.
\end{itemize}
\end{definition}
Assuming $C$ has a propelinear strucure, for all $x,y \in C$, denote by $*$ the binary operation such that $x*y=x+\pi_x(y)$. Then, $(C,*)$ is a group, which is not abelian in general.
The vector $\be$ is always a codeword and $\pi_\be=I$ is the identity permutation. Hence, $\be$ is the identity element in $C$ and $x^{-1}=\pi^{-1}_x(x)$, for all $x\in C$. We call $(C,*)$ a \emph{propelinear code}.
Note that we can extend $*$ to $\F^n$ as an action of $C$ over $\F^n$, then for all $x\in C$ and for all $y\in \F^n$ we have that $x*y=x+\pi_x(y) \in \F^n$. Henceforth we use $xy$ instead of $x*y$ if there is no confusion.

An automorphism of a binary code $C$ is a permutation on the set of coordinates leaving the code invariant. We denote by $Aut(C)$ the set of all automorphisms of $C$.

\begin{proposition}\cite{mog}\label{aut}
Let $C$ be a propelinear code. Then:
\begin{enumerate}[(i)]
\item For $x\in C$ we have $x\in K(C)$ if and only if $\pi_x\in Aut(C)$.
\item The kernel $K(C)$ is a subgroup of $C$ and also a binary linear space.
\item If $c\in C$ then $\pi_c\in \Aut(K(C))$.
\end{enumerate}
\end{proposition}
Let $x$ be in a propelinear code, we denote the element $x * x^{i-1}$ by $x^i$, with $x^1=x$, for any $i > 1$.
\begin{lemma}\label{xi}
Let $C$ be a propelinear code. Then $x^i=x+\pi_x(x)+\ldots+\pi_x^{i-1}(x)$, for all $x\in C$.
\end{lemma}
\begin{proof}
We proceed by induction on $i$. The base case is $x^2=x+\pi_x(x)$. Let us see the inductive step, $x^i=x*x^{i-1}=x+\pi_x(x^{i-1})=x+\pi_x(x+\pi_x(x)+\ldots+\pi_x^{i-2}(x))=x+\pi_x(x)+\ldots+\pi_x^{i-1}(x)$.
\end{proof}

\begin{definition}\cite{rs}
An Hadamard full propelinear code ($\HFP$-code) is an Hadamard propelinear code $C$ such that for every $\ba \in C$, $\ba \neq \be$, $\ba \neq \bu$, the permutation $\pi_\ba$ has not any fixed coordinate and $\pi_\be = \pi_\bu = I$.
\end{definition}

\begin{definition}
The associated group $\Pi$ of a propelinear code $C$ is the set of permutations of all elements of $C$, $\{\pi_x\in S_n : x\in C\}$, which is a group with the composition.
\end{definition}

\begin{proposition}\label{perm}
Let $C$ be an Hadamard full propelinear code. Then $\bu \in K(C)$ and the associated group of C is isomorphic to $C/\langle \bu \rangle$.
\end{proposition}
\begin{proof}
Let $x$ be a codeword in $C$. Since $x+\bu  = x+\pi_x(\bu)=x * \bu \in C$, we have $\bu \in K(C)$. Let $\varphi :C \rightarrow \Pi$ be the mapping given by $\varphi(x)=\pi_x$ for all $x\in C$. As $\pi_{x * y}=\pi_x\pi_y$ for all $x,y\in C$, we have that the mapping $\varphi$ is a group homomorphism. Since $C$ is full propelinear, the kernel of this homomorphism is $\langle \bu \rangle$. Thus $C/\langle \bu \rangle \simeq \varphi(C)=\Pi$.
\end{proof}

\begin{definition}
An extension of a group $H$ by a group $N$ is a group $G$ with a normal subgroup $M$ such that $M \simeq N$ and $G / M \simeq H$. This information can be encoded into a short exact sequence of groups
\vspace{-0.3cm}

$$1\rightarrow N \rightarrow G \rightarrow H \rightarrow 1.$$
\end{definition}

Let $C$ be an Hadamard full propelinear code with associated group $\Pi$, and from \Cref{perm} we have $C/ \Pi \simeq C_2$. Thus, an Hadamard full propelinear code is a central extension of the associated group by the cyclic group of order 2. The search for Hadamard full propelinear codes is an extension problem with the following short exact sequence
\vspace{-0.3cm}

$$1\rightarrow C_2 \rightarrow C \rightarrow \Pi \rightarrow 1.$$

\begin{lemma}\cite{rs}\label{cyclic}
Let $(C,*)$ be an Hadamard propelinear code of length 4t. Then
C is not a cyclic group of order 8t.
\end{lemma}
A $n \times n$ matrix whose rows are cyclic shift versions of a vector $c\in \F^n$ is called \textit{circulant matrix}. The binary code consisting of the rows of a binary circulant normalized Hadamard matrix and their complements is called a circulant Hadamard code.
\begin{lemma}\cite{rifacirculant}\label{rifacirculant}
Let $C$ be a circulant Hadamard code of length $4t$, then $C$ is an $\HFP$-code with a cyclic associated group $\Pi$ of order $4t$. Vice versa, an $\HFP$-code $C$ with a cyclic associated group $\Pi$ of order $4t$ is an Hadamard circulant code of length $4t$.
\end{lemma}

\begin{lemma}\cite{rifacirculant}\label{rifacirculantk1}
Let $C$ be a nonlinear circulant Hadamard code of length $4t$. Then the dimension of the kernel is $k=1$.
\end{lemma}
\begin{remark}\label{remarkTuryn}
No circulant Hadamard matrices of order larger than $4$ has ever been found, but the nonexistence is still a non proven result.  The main work on Ryser's conjecture \cite{rys} about circulant Hadamard matrices seems to be due to Turyn~\cite{tur}. He showed that a circulant Hadamard matrix of length $4t$ fulfils that $t$ is an odd square.
\end{remark}
\begin{conjecture}\label{circHadConj}
There do not exist circulant Hadamard codes of length $4t$ for $t>1$.
\end{conjecture}

\begin{proposition}\label{HDS}
Let $C$ be an Hadamard full propelinear code of length $4t$. If $C \simeq G \times \langle \bu \rangle$, then there exists an Hadamard difference set in $G$.
\end{proposition}
\begin{proof}
From \cite{itoII}, if $C=G \times \langle \bu \rangle$ is an Hadamard group of order $8t$, then there exists an Hadamard difference set in $G$. From \cite{rs}, $C$ is an Hadamard full propelinear code if and only if $C$ is an Hadamard group.
\end{proof}

Next lemmas about the rank of a code and the dimension of its kernel are well known. We omit the proofs.
\begin{lemma}\cite{Rifa2017}\label{brk}
Let $C$ be a code of length $4t=2^st'$, where $t'$ is odd. The rank $r$ of $C$ fulfils $r\leq \frac{2^{s+1}t'}{2^k}+k-1$, where $k$ is the dimension of the kernel.
\end{lemma}
\begin{lemma}\cite{prv3}\label{odd}
Let $C$ be a nonlinear Hadamard code of length $2^st'$, where $t'$ is odd. The dimension of the kernel $k$ fulfils $1 \leq k \leq s-1$.
\end{lemma}
\begin{lemma}\cite[Thm. 2.4.1 and Thm. 7.4.1]{ak}\label{aske}
Let $C$ be an Hadamard code of length $4t=2^st'$, where $t'$ is odd.
\begin{enumerate}[(i)]
\item If $s\geq 3$ then the rank of $C$ is $r\leq 2t$, with equality if $s=3$.
\item If $s=2$ then $r=4t-1$.
\end{enumerate}
\end{lemma}

In the whole paper, if there appears some codeword $v$ of a code $C$ of length $4t$, that is presented in the following way $v=(\alpha,\beta,\gamma,\delta,\ldots)$, then each element that compounds $v$ has the same length. For example, if $v=(\alpha_1,\alpha_2,\ldots,\alpha_n)$, then the length of each $\alpha_i$ is $4t/n$.

To finish this section, we recall a notation introduced in \cite{BBR} to denote the Hadamard full propelinear codes that we will use in the next section. An $\HFP(t,2,2,2_\bu)$-code means an Hadamard full propelinear code of type $C_t\times C_2 \times C_2 \times C_2$ where the codeword $\bu$ is the generator of the last $C_2$. So the numbers in parentheses mean the orders of the cyclic groups. If the parameter in the parentheses is $Q$, then it means the quaternion group of eight elements.

\section{Associated group $C_{2t} \times C_2$} \label{C2tC2}
Now, we introduce a subclass of Hadamard full propelinear codes whose group structure consists of direct product of groups, fulfilling that its associated group $\Pi$ is $C_{2t} \times C_2$. In other words, we study the short exact sequence
\vspace{-0.3cm}

$$1\rightarrow C_2 \rightarrow C \rightarrow C_{2t}\times C_2 \rightarrow 1,$$
where $C$ is a nontrivial direct product of groups.

\begin{proposition}\label{HFPeven2t2} \label{HFP2t2}
Let $C$ be an Hadamard full propelinear code of lenght $4t$ with associated group $C_{2t}\times C_2$. If $C$ as a group is a nontrivial direct product, then $C$ is some of the following $\HFP$-codes:
\begin{enumerate}[i)]
\item $\HFP(4t_\bu,2)\simeq C_{4t} \times C_2 = \langle \ba, \bb \mid \ba^{4t}=\bb^2=\be, \ba^{2t}=\bu \rangle$. [\Cref{4tu2}]
\item $\HFP(2t,2,2_\bu) \simeq C_{2t}\times C_2 \times C_2 = \langle \ba, \bb, \bu \mid \ba^{2t}=\bb^2=\be \rangle$. [\Cref{2t22u}]
\item $\HFP(2t,4_\bu) \simeq C_{2t} \times C_4 = \langle \ba, \bb \mid \ba^{2t}=\bb^4=\be, \bb^2=\bu \rangle$. [\Cref{2t4u}]
\item $\HFP(t,Q_\bu) \simeq C_t \times Q= \langle \bd,\ba,\bb \mid \bd^t=\ba^4=\bb^4=\be, \ba^2=\bu^2=\bu, \ba\bb\ba=\bb \rangle$, where $Q$ is the quaternion group of eight elements. [\Cref{tQu}] 
\end{enumerate}
\end{proposition}
\begin{proof}
Note that there are two different cases depending on the parity of the value of $t$. Firstly, we suppose that $t$ is odd, so $C_{2t} \times C_2 \simeq C_t \times C_2^2$. Let $E$ be an HFP-code with $\Pi=C_t \times C_2^2$. From \Cref{perm} we have that $\Pi=E/\langle \bu \rangle$. Thus, the code $E$ is an extension of $C_t\times C_2^2$ by $\langle \bu \rangle \simeq C_2$.
From \cite[Table 2]{flan} we have that the central extensions of $C_t\times C_2^2$ by $C_2$ with $t$ odd are $C_t\times C_2^3$, $C_t \times C_4 \times C_2$, $C_t \times Q$ and $C_t \times D$ (where $D$ is the dihedral group of order $8$). Furthermore from \cite[Prop. 6]{ito}, if $t$ is odd we have that $C_t \times D$ cannot be an Hadamard group. Thus, from \cite{rs} there are no $\HFP$-codes with the group structure $C_t \times D$.

Now, we suppose that $t$ is even. Let $E$ be an HFP-code with associated group $\Pi=C_{2t} \times C_2$. From \Cref{perm} we have that $\Pi=E/\langle \bu \rangle$. Thus, the code $E$ is an extension of $C_{2t}\times C_2$ by $\langle \bu \rangle \simeq C_2$.
The extensions of $C_2$ by $C_2$ are $C_4$ and $C_2\times C_2$. We denote by $E_1$ any of these extensions. Making the direct product $E_1 \times C_{2t}$ we get two extensions of $C_{2t}\times C_2$ by $C_2$. Thus, we obtain $\HFP(2t,4_\bu)$ and $\HFP(2t,2,2_\bu)$.

Let $t=2^s t'$ with $t'$ odd, so $C_{2t} \times C_2$ is isomorphic to $C_{2^{s+1}} \times C_{t'} \times C_2$. Let $E_2$ be the extensions of $C_{2^{s+1}}$ by $C_2$. As $E_2 / C_2$ is cyclic we have $E_2$ is abelian. Thus, $E_2$ is $C_{2^{s+2}}$ or $C_{2^{s+1}} \times C_2$. Making the direct product $E_2 \times C_{t'} \times C_2$ we get two extensions of $C_{2t}\times C_2$ by $C_2$. Thus, we obtain $\HFP(4t_\bu,2)$ and $\HFP(2t,2,2_\bu)$.
Let $E_3$ be the extension of $C_{t'}$ by $C_2$, which is abelian since $E_3/C_2$ is cyclic. Hence, $E_3$ is $C_{2t'}$. Therefore, the direct product $C_{2t'} \times C_{2^{s+1}} \times C_2 \simeq C_{2t} \times C_2^2$ is an extension of $C_{2t}\times C_2$ by $C_2$, which corresponds with an $\HFP(2t,2,2_\bu)$-code.
\end{proof}

\begin{remark}
In the conditions of \Cref{HFP2t2}, if $t$ is even, then $C$ cannot be an $\HFP(t,Q_\bu)$-code because its associated group is $C_t \times C_2^2$.
If the value of $t$ is odd, then $\HFP(t,2,2,2_\bu) \simeq \HFP(2t,2,2_\bu)$ and $\HFP(t,4_\bu,2) \simeq \HFP(2t,4_\bu) \simeq \HFP(4t_\bu,2)$. 
\end{remark}

\begin{proposition}\label{rankkernel}
Let $C$ be a nonlinear Hadamard full propelinear code of length $4t$ with associated group $C_{2t} \times C_2$. Then:
\begin{enumerate}[i)]
\item If $t$ is odd, then $r=4t-1$ and $k=1$.
\item If $t$ is even, then $r\leq 2t$, and $r = 2t$ if $t \equiv 2 \modu 4$.
\end{enumerate}
\end{proposition}
\begin{proof}

The first and the second item are easy from \cite{Rifa2017} and \Cref{aske}, respectively.
\end{proof}

\begin{proposition} \label{genpermutations}
Let $C =\langle \ba, \bb, \bu \rangle$ be a code of type $\HFP(2t,2,2_\bu)$ or $\HFP(2t,4_\bu)$ or $\HFP(4t_\bu,2)$. Then, up to equivalence, we have
\begin{enumerate}[i)]
\item $\pi_\ba=(1,2,\ldots,2t)(2t+1,2t+2,\ldots,4t)$,
\item $\pi_\bb=(1,2t+1)(2,2t+2)\ldots (2t,4t)$,
\item Knowing the value of $\ba$ is enough to define $\bb$,
\item $\Pi=C_{2t}\times C_2$.
\end{enumerate}
\end{proposition}
\begin{proof}
In any case, we have that $\ba^{2t}, \bb^2 \in \{\be,\bu\}$, so $\pi_\ba$ has order $2t$, and $\pi_\bb$ has order $2$. As $\pi_\ba$ has order $2t$, then $\pi_\ba$ is the product of two cycles of length $2t$.
Indeed, if we have a cycle of length $j<2t$ then $\pi_{\ba^j}=\pi_\ba^j$ has a fixed point, which contradicts that $C$ is full propelinear. Without loss of generality we can set $\pi_\ba=(1,2,\ldots,2t)(2t+1,2t+2,\ldots,4t)$, so we have the first item. As $\pi_\bb$ has order 2, then $\pi_\bb$ is a product of disjoint transpositions.
Each one of the transpositions sends an element of the first half of $\{1,2,\ldots,4t\}$ to the second half and vice versa.
Indeed, assume for instance than $\pi_\bb$ moves the first position to the $i$th, where $i < 2t$, which is the same position that we obtain using $\pi_{\ba^{i-1}}$, so $\pi_{\bb^{-1}\ba^{i-1}}$ has a fixed point which contradicts that $C$ is full propelinear.
Furthermore, if we assume that $\pi_\bb$ moves, for instance, the first position to the $i$th position in the second half of $\{1,2,\ldots,4t\}$ then $\pi_\bb$ is uniquely determined.
Indeed, as $\pi_\ba \pi_\bb=\pi_\bb \pi_\ba$ we have that $2 \rightarrow 1$ by $\pi_{\ba^{-1}}$, $1\rightarrow i$ by $\pi_\bb$ and $i\rightarrow i+1$ by $\pi_\ba$.
Hence, $2\rightarrow i+1$ by $\pi_\bb$, and so on. Thus, we can assume $\pi_\bb=(1,2t+1)(2,2t+2)\ldots (2t,4t)$, and we have the second item.
Since $\ba\bb=\bb\ba$, we have $\bb=\pi_\ba(\bb)+\ba+\pi_\bb(\ba)=\pi_\ba(\bb)+\widehat{\ba}$, where $\widehat{\ba}=\ba+\pi_\bb(\ba)=(\widehat{a}_1,\ldots,\widehat{a}_{2t},\widehat{a}_1,\ldots,\widehat{a}_{2t})$, then $b_i=\sum_{j=i+1}^{2t}\widehat{a}_j$ for $i\in \{1,\ldots,2t-1\}$ and $b_{2t}\in \{0,1\}$. We know that $\bb^2 \in \{\be, \bu\}$. If $\be=\bb^2$, then we have that $\bb=\pi_\bb(\bb)$, so $b_i=b_{2t+i}$ for $i\in\{1,\ldots,2t\}$. Thus

$$
\begin{array}{rrl}
\bb & = & \displaystyle \Biggl(b_{2t}+\sum_{j=2}^{2t}\widehat{a}_j,b_{2t}+\sum_{j=3}^{2t}\widehat{a}_j,\ldots,b_{2t}+\widehat{a}_{2t},b_{2t},\\[0.2cm]
 & & \hspace{0.3cm}\displaystyle b_{2t}+\sum_{j=2}^{2t}\widehat{a}_j,b_{2t}+\sum_{j=3}^{2t}\widehat{a}_j,\ldots,b_{2t}+\widehat{a}_{2t},b_{2t} \Biggr).
\end{array}
$$
If $\bb^2=\bu$, then $\bb=\pi_\bb(\bb)+\bu$, so $b_i=b_{2t+i}+1$ for $i\in\{1,\ldots,2t\}$. Thus

$$
\begin{array}{rrl}
\bb & = & \displaystyle \Biggl(b_{2t}+\sum_{j=2}^{2t}\widehat{a}_j,b_{2t}+\sum_{j=3}^{2t}\widehat{a}_j,\ldots,b_{2t}+\widehat{a}_{2t},b_{2t},\\[0.2cm]
 & & \hspace{0.3cm}\displaystyle 1+b_{2t}+\sum_{j=2}^{2t}\widehat{a}_j,1+b_{2t}+\sum_{j=3}^{2t}\widehat{a}_j,\ldots,1+b_{2t}+\widehat{a}_{2t},1+b_{2t} \Biggr).
\end{array}
$$
From a computational point of view, this third item saves us computing time. As knowing the value of the generator $\ba$, we know the value of $\bb$. Therefore, by brute-force search, we only need to vary the generator $\ba$.\\
Note that the fourth item is immediate from \Cref{perm}.
\end{proof}
In the Sections \ref{4tu2}, \ref{2t22u}, and \ref{2t4u} we will use the permutations associated to the generators as in \Cref{genpermutations}.

\begin{theorem}
Let $C=\langle \ba, \bb, \bu \rangle$ be a code of type $\HFP(2t,2,2_\bu)$ or $\HFP(2t,4_\bu)$ or $\HFP(4t_\bu,2)$. If $C$ is nonlinear, then the dimension of the kernel of $C$ is $k \leq 3$.
\end{theorem}
\begin{proof}
The kernel of $C$ is a linear code $K$ of dimension $k$ with a generator matrix $G_k$,
\vspace{-0.2cm}

$$G_k=
\left(\begin{matrix}
\bu \\
v_2 \\
\vdots \\
v_k 
\end{matrix}\right),$$
for some codewords $v_2,v_3,\ldots,v_k \in C$ with weight $2t= 2^{k-2} \lambda$, where $\lambda$ is the number of cosets. From \Cref{odd} and $C$ is nonlinear, $\lambda$ is even and greater than $2$. 
We can set $K'$ a constant-weight code of dimension $k-1$, with generator matrix

\vspace{-0.2cm}

$$G_{k-1}=
\left(\begin{matrix}
v'_2 \\
\vdots \\
v'_k 
\end{matrix}\right),$$

where there are no all-zero columns. Note that $K'$ is built by removing firstly the all-one row, and secondly the all-zero columns. Hence we are in the conditions of \cite[Thm. 7.9.5]{hp}, so we have  $K'$ is equivalent to the $\lambda$-fold replication of a simplex code.
Thus, the generator matrix $G_{k-1}$ is equivalent to the matrix
\vspace{-0.2cm}

$$
\left(\begin{matrix}
S_{k-1}^{(1)} & S_{k-1}^{(2)} & \cdots & S_{k-1}^{(\lambda)}
\end{matrix}\right)
,$$

where $S_{k-1}^{(i)}$ are the generator matrices of simplex codes of dimension $2^{k-1}$, with $i\in \{1,2,\ldots,\lambda\}$. As the weight of the rows of each simplex $S_{k-1}^{(i)}$ is $2^{k-2}$ and the length is $2^{k-1}-1$, for recovering the code $K$ we have to add exactly $\lambda$ all-zero columns. In fact, $4t=(2^{k-1}-1)\lambda+\alpha$ where $\alpha$ is the number of all-zero columns, so $\alpha=\lambda$. 
Denote by $K_1$ the set of codewords of $K$ that have $0$'s in the coordinates corresponding to the removed columns of $G_k$, and $K_2$ is the complement of $K_1$. Note that $K=K_1 \cup K_2$.
We can express the generator matrix of the kernel as,
\vspace{-0.2cm}

$$G_k=
\left(\begin{matrix}
\bu_{2^{k-1}} & \bu_{2^{k-1}} & \cdots & \bu_{2^{k-1}} \\
A_1 & A_2 & \cdots & A_\lambda
\end{matrix}\right)
,$$

for some blocks of coordinates $A_i$ with $i\in \{1,2,\ldots,\lambda\}$.
Omitting the first row, let's see that there are $2^{k-1}-1$ columns between an all-zero column and the next all-zero column in the same cycle of $\pi_\ba=(1,2,\ldots,2t)(2t+1,\ldots,4t)$. Note that \emph{in the same cycle of }$\pi_\ba$ means that the first column is the next column to the $2t$-th column. We suppose that the column $i$ and $i+j$ are all-zero columns and there are no all-zero columns between them, for some $i,j$ with $j < 2^{k-1}-1$.
We denote by $v[i]$ the $i$-th coordinate of $v$.
From \Cref{aut}, we have $\pi_\ba^m \in \Aut(K)$ since $\ba^m \in C$ for any $m$.
Thus, we have $\pi_\ba^m(v_l) \in K$ for $l\in \{2,\ldots,k\}$ and for any $m$. If $v_l[i-m]=x_{l,m}$ and $v_l[i+j-m]=y_{l,m}$, then $\pi^m_\ba(v_l)[i]=x_{l,m}$ and $\pi^m_\ba(v_l)[i+j]=y_{l,m}$.
As $\pi_\ba^m(v_l) \in K$ for any $m$, we have $x_{l,m}=y_{l,m}$.
Indeed, if $x_{l,m}=0$ and $y_{l,m}=1$, then $\pi_\ba^m(v_l) \in K_1$ since $\pi_\ba^m(v_l)[i]=0$, and $\pi_\ba^m(v_l) \in K_2$ since $\pi_\ba^m(v_l)[i+j]=1$, which is a contradiction.
Hence, we have the same blocks of coordinates between all-zero columns.
Therefore, if $j$ is less than $2^{k-1}-1$ contradicts that there are exactly $\lambda$ all-zero columns.
Thus, $A_1=A_2=\ldots=A_\lambda$ and there is an all-zero column in each $A$-block.
Note that there are no repetitions of columns in each $A$-block, since $G_{k-1}$ is a $\lambda$-fold replication of simplex.
Therefore, the generator matrix of the kernel has blocks with the following shape
\vspace{-0.2cm}

$$\left(\begin{matrix}
\bu_{4t/\lambda}\\
S_{k-1} \mid \zero \\
\end{matrix}\right),$$
where $\zero$ is the all-zero column.
The code generated by the previous block matrix is a cyclic linear Hadamard code of dimension $k$, which is the dual of a cyclic extended Hamming code. Note that the extended Hamming code is never cyclic, except for length 4. Thus, we have $k \leq 3$.  
\end{proof}

The following lemma will be useful in some proofs through the next sections. It is an improved version of a lemma which appears in \cite{Rifa2017}.
\begin{proposition}\label{parseval2}
Let $C$ be an Hadamard code of length $4t$ with dimension of the kernel $k$, and $s \in K(C)\setminus \langle \bu \rangle$. Then $C_s$ consists of two copies of an Hadamard code of length $2t$ and dimension of the kernel equal to $k-1$, where $C_s$ is the projection from $C$ onto $\Supp(s)$.
\end{proposition}
\begin{proof}
In \cite{Rifa2017}, it is proved that $C_s$ consists of two copies of an Hadamard code of length $2t$. As we project the code $C$ over the support of $s$, we have that $s_s=\bu_{2t}$. Thus, the kernel of $C_s$ decreases in one unit.
\end{proof}

\subsection{\boldmath{$\HFP(4t_\bu,2)$}-codes} \label{4tu2}
In this subsection we assume that $C=\langle \ba,\bb : \bb^2=\bu^2=\be, \ba^{2t}=\bu \rangle$. 
\begin{proposition}\label{teven-4tu2}
Let $C$ be an $\HFP(4t_\bu,2)$-code of length $4t$ with $t > 1$. Then $t$ is even.
\end{proposition}
\begin{proof}
Suppose that $t$ is odd. Thus, $C_{4t} \times C_2 \simeq C_{2t} \times C_{4}$ and an $\HFP(4t_\bu,2)$-code is equivalent to an $\HFP(2t,4_\bu)$-code. Without loss of generality, we suppose that $C=\langle \ba, \bb\rangle$ where $\ba$ has order $2t$, $\bb$ has order $4$ and $\bb^2=\bu$. From \Cref{rankkernel}, $r=4t-1$.
We know that the rank of $C$ is $r \leq \text{rank}(H)+1$ (due to the vector $\bu$) but $r \leq \text{rank}(H)$ if $\bu$ is a combination of rows of $H$, where

$$H=\left( \begin{array}{l}
\ba \\
\ba^2 \\
\vdots \\
\ba^{2t-1} \\
\be \\
\bb \\
\bb\ba \\
\bb\ba^2 \\
\vdots \\
\bb\ba^{2t-1}
\end{array}\right)=
\left( \begin{array}{l}
\ba \\
\ba + \pi_\ba(\ba)\\
\vdots\\
\ba + \pi_\ba(\ba^{2t-2})\\
\ba+ \pi_\ba(\ba^{2t-1})\\
\bb \\
\bb + \pi_\bb(\ba) \\
\bb + \pi_\bb(\ba^2) \\
\vdots \\
\bb + \pi_\bb(\ba^{2t-1})
\end{array}\right).$$
If we sum the first half of rows, then we obtain 
$\ba+\ba^2+\ldots+\ba^{2t-1}=
\pi_\ba(\ba)+\ldots+\pi_\ba(\ba^{2t-1})=
\pi_\ba(\ba+\ldots+\ba^{2t-1})$.
Thus $\ba+\ba^2+\ldots+\ba^{2t-1} = w$, such that $w=\pi_\ba(w)$. Therefore $w \in \{\be_{4t}, \bu_{4t}, (\be_{2t},\bu_{2t}), (\bu_{2t},\be_{2t})\}$.
If $w=\be_{4t}$, then in the first half of rows there is at most $2t-2$ independent rows, but if $w\in\{\bu_{4t},(\be_{2t},\bu_{2t}),(\bu_{2t},\be_{2t})\}$, then there is at most $2t-1$ independent rows.
Making the sum of the second half of rows we obtain $\pi_\bb(\ba+\ba^2+\ldots+\ba^{2t-1})$, so the number of independent rows in the second half of rows is at most equal to the number of independent rows in the first half plus one, due to the vector $\bb$.
Therefore, if $w=\be_{4t}$ the rank of $H$ is at most $4t-3$, and so $r \leq 4t-2$. But if $w=\bu_{4t}$ then $\bu_{4t}$ appears as combination of the rows of the first half and also as combination of the rows of the second half, then the rank of the matrix $H$ is at most $4t-2$, and so $r \leq 4t-2$.
If $w\in \{(\be_{2t},\bu_{2t}),(\bu_{2t},\be_{2t})\}$, then the sum of all rows of $H$ is the vector $\bu_{4t}$ and so in each column of $H$ there is an odd amount of ones, which contradicts that $H$ is an Hadamard matrix for $t > 1$. Therefore, $r<4t-1$ which contradicts that $t$ is odd by \Cref{rankkernel}.
\end{proof}

\begin{proposition}\label{rank2t}
Let $C$ be a nonlinear $\HFP(4t_{\bu},2)$-code. If $t$ is a power of two, then $r=2t$.
\end{proposition}
\begin{proof}
We suppose that $t$ is a power of two, then $4t=2^s$. Thus, the rank of $C$ is greater than or equal to the rank of the matrix generated by $\ba =(a_1,a_2)$, where $a_1$ and $a_2$ are the first and the second half of components, respectively, of the generator $\ba$. 
Vectors in $C$ can be written as polynomials in $GF(2)[x]/(x^n-1)$, where the coordinates of the vector have been substituted by the coefficients of the polynomial.
Since $\ba^{2t}=\bu$, we have that $\wt(a_i)$ is odd, so $a_i$ does not contain the factor $(x-1)$, for $i=1,2$. Let $A$ be the matrix whose rows are $\ba,\ba^2,\ldots,\ba^{2t}$, note that $rank(A)=2t-\deg(\Gcd(a_1,a_2,x^{2^{s-1}}-1))$. Since $a_1$ and $a_2$ do not contain the factor $(x-1)$, and $x^{2^{s-1}}-1 = (x-1)^{2^{s-1}}$, we have that $\deg(\Gcd(a_1,a_2,x^{2^{s-1}}-1))=0$. Thus $rank(C) \geq rank(A)=2t$. From \Cref{rankkernel}, we have that $r\leq 2t$, so $r=2t$.
\end{proof}

\begin{proposition}\label{b-4tu2}
Let $C$ be an $\HFP(4t_{\bu},2)$-code. Then the generator $\bb$ does not belong to the kernel of $C$.
\end{proposition}
\begin{proof}
We are going to show that $\bb\notin K(C)$. Assume the contrary and take $\ba=(a_1,a_2)$. We know that $\bb=(\beta,\beta)$, where $\wt(\beta)=t$. If $\bb\in K(C)$ then $\bb\ba,\bb+\ba \in C$ so $(\beta+a_2,\beta+a_1),(\beta+a_1,\beta+a_2) \in C$ and both vectors should be at distance $2t$, so $\wt(a_1+a_2)=t$. Indeed, if $d(\ba\bb,\ba+\bb)=0$, then $\bb=\pi_\ba(\bb)$ so $\bb=\bu$ or $\be$, which is impossible. If $d(\ba\bb,\ba+\bb)=4t$ then $\ba=(a_1,a_1\bu)$ and $\bb=\bomega_{4t}$ (or its complement). We denote, for any $j$, the first and the second half of $\ba^j$ by $a^j_1$ and $a^j_2$, respectively. 
We have that $\ba^{2i}=(a^{2i}_1,a^{2i}_1)$ for $i\in\{1,\ldots,t\}$, then the projection of the first half $a^{2i}_1$ onto the support of $\beta$, $a^{2i}_1|_\beta$, conform an Hadamard matrix of length $t$, and $a^2_1|_\beta$ generates a cyclic group, which is impossible from \Cref{cyclic}. 
Note that this lemma is only applicable if the parameter $t$ is greater than or equal to $4$, but if $t=3$ it is clear that $k=1$.
Thus, $d(\ba\bb,\ba+\bb)=2t$ and $\wt(a_1+a_2)=t$.
Now, using the same argument for all elements $\ba^i$ of $C$ instead of $\ba$, we obtain $\wt(a_1^i+a_2^i)= t$ and also the distance from $(a_1^i+a_ 2^i)$ to $(a_1^j+a_2^j)$ is $t$. Indeed, $d\big((a_1^i+a_ 2^i),(a_1^j+a_2^j)\big)=\wt(a_1^i+a_1^j+a_ 2^i+a_2^j)$, and using \Cref{xi}, $\wt(a_1^i+a_1^j+a_ 2^i+a_2^j)=\wt(\pi_{\ba^i}(a_1^{j-i}+a_2^{j-i}))=\wt(a_1^{j-i}+a_2^{j-i})=t$. Hence, the elements $(a_1^i+a_2^i)$, for $i\in\{1,\ldots,4t\}$, give a cyclic Hadamard code of length $2t$, which is impossible. Therefore, $\bb \not\in K(C)$.
\end{proof}

\begin{proposition}\label{eu-4tu2}
Let $C$ be an $\HFP(4t_\bu,2)$-code. The vectors $(\be_{2t},\bu_{2t})$, $\bomega_{4t}$, and $(\bomega_{2t},\bomega_{2t}\bu)$ are not in the kernel of $C$ for $t\geq 4$.
\end{proposition}
\begin{proof}
Firstly, let $(\be_{2t},\bu_{2t})$ be in $K(C)$. Note that $(\be_{2t},\bu_{2t})= \ba^t\bb$. Indeed, if $(\be_{2t},\bu_{2t})=\ba^i$ for some $i$, we have that $(\be_{2t},\bu_{2t})^2=\be$ which is not possible. If $(\be_{2t},\bu_{2t})= \ba^i\bb$ for some $i$, we have that $(\be_{2t},\bu_{2t})^2=\bu$, so $i=t$. 
From \Cref{parseval2}, the projection of $C$ over the support of $(\be_{2t},\bu_{2t})=\ba^t\bb$ consists of two copies of an Hadamard code. 
We denote by $x|_{(\be_{2t},\bu_{2t})}$ the projection of a vector $x$ over the support of $(\be_{2t},\bu_{2t})$.
As $\bb^2=\be$ and $\bb \,\ba^t\bb=\ba^t\bb\,\bb$, we have that $\bb=(\beta,\beta\bu,\beta,\beta\bu)$ for some $\beta$.
Thus, $\ba^t=(\beta\bu,\beta,\beta,\beta\bu)$ and $\ba^t|_{(\be_{2t},\bu_{2t})}=\bb|_{(\be_{2t},\bu_{2t})}$.
As $\ba^i\bb\,\ba^t\bb=\ba^t\bb\,\ba^i\bb$, we have that $\ba^i\bb=(\alpha_{i,1}\bu,\alpha_{i,2}\bu,\alpha_{i,2},\alpha_{i,1})$, for some $\alpha_{i,j}$ where $i\in \{1,\ldots,2t\}$ and $j \in \{1,2\}$. 
Therefore, $\ba^{t+i}=(\alpha_{i,1},\alpha_{i,2},\alpha_{i,2},\alpha_{i,1})$ and $\ba^i\bb|_{(\be_{2t},\bu_{2t})}=\ba^{t+i}|_{(\be_{2t},\bu_{2t})}$.
Let $\alpha$ be the vector $\ba|_{(\be_{2t},\bu_{2t})}$ of length $2t$, and let $\pi_\alpha=(1,\ldots,2t)$. Note that $\alpha^{2t}=\ba^{2t}|_{(\be_{2t},\bu_{2t})}=\bu_{2t}$. We have that $\langle \alpha \rangle$ is a cyclic $\HFP$-code of length $2t$, which contradicts \Cref{cyclic} for $t\geq 2$.

Now, we suppose that $\bomega_{4t}\in K(C)$. We have that $\bomega_{4t} \in \{ \ba^i, \ba^i\bb \}$ for some $i$. If $i$ is even, then $\bomega_{4t}^2=\be$, so $\bomega_{4t} \in \{\bb,\bb\bu\}$ which contradicts \Cref{b-4tu2}. If $i$ is odd, then $\bomega_{4t}^2=\bu$, so $i=t$, which contradicts \Cref{teven-4tu2}.

Finally, let $(\bomega_{2t},\bomega_{2t}\bu)$ be in $K(C)$. We will see that $(\bomega_{2t},\bomega_{2t}\bu)=\ba^t\bb$.
We suppose that $(\bomega_{2t},\bomega_{2t}\bu)=\ba^i$ for some $i$. If $i$ is even, then $(\bomega_{2t},\bomega_{2t}\bu)^2=\be$, which is not possible. If $i$ is odd, then $(\bomega_{2t},\bomega_{2t}\bu)^2=\bu$, so $i=t$ which contradicts \Cref{teven-4tu2}.
Now, we suppose that $(\bomega_{2t},\bomega_{2t}\bu)=\ba^i\bb$ for some $i$. If $i$ is odd, then $(\bomega_{2t},\bomega_{2t}\bu)^2=\be$ which is not possible. If $i$ is even, then $(\bomega_{2t},\bomega_{2t}\bu)^2=\bu$, so $i=t$. Hence, $\ba^t\bb=(\bomega_{2t},\bomega_{2t}\bu)$.
As $\ba^{2i}\ba^t\bb=\ba^t\bb \ba^{2i}$, we have that $\ba^{2i}=(\alpha_{2i,1},\alpha_{2i,2},\alpha_{2i,2},\alpha_{2i,1})$ for some $\alpha_{2i,j}$, with $i\in \{1,\ldots,t\}$ and $j\in \{1,2\}$. From \Cref{parseval2}, the projection of $C$ over the support of $(\bomega_{2t},\bomega_{2t}\bu)$ has an Hadamard structure. Taking the first half of coordinates of $\ba^{2i}$ and projecting it over the support of $\bomega_{2t}$ we obtain a cyclic Hadamard code of length $t$, which contradicts \Cref{cyclic} for $t\geq 4$.
\end{proof}

\begin{proposition}\label{kneq2}
Let $C$ be an $\HFP(4t_\bu,2)$-code. Then $k\neq 2$ for $t\geq 4$.
\end{proposition}
\begin{proof}
We suppose that $k=2$. Hence, there exists a vector $v \in K(C)\setminus \langle \bu \rangle$. From \Cref{aut}, as $\ba \in C$, we have that $\pi_\ba(v)\in K(C)$. If $\pi_\ba(v) \in \langle \bu \rangle$, then it contradicts $v \notin \langle \bu \rangle$. If $\pi_\ba(v) = v$, then $v=(\be_{2t},\bu_{2t})$, which contradicts \Cref{eu-4tu2}. If $\pi_\ba(v)=v\bu$ then $v \in \{ \bomega_{4t},(\bomega_{2t},\bomega_{2t}\bu)\}$, which contradicts \Cref{eu-4tu2}.
\end{proof}

\begin{lemma}\label{evenpi-4tu2}
Let $C$ be an $\HFP(4t_\bu,2)$-code. If $v\in K(C) \setminus \langle \bu \rangle$, and $j_v$ is the smallest value such that $\pi_\ba^{j_v}(v)=v$, then $j_v$ is even.
\end{lemma}
\begin{proof}
Let $v \in K(C)\setminus \langle \bu \rangle$.
As $\ba^i \in C$ for $i=1,\ldots,2t$, from \Cref{aut} we have that $\pi_\ba^i \in \Aut(K(C))$, so $\pi_\ba^i(v)\in K(C)$. Therefore, there is some $j$ such that $\pi_\ba^j(v)=v$. We note that $j\neq 1$. Indeed, $j=1$ implies that $\pi_\ba(v)=v$, so $v\in \{ \be,\bu,(\be_{2t},\bu_{2t}),(\bu_{2t},\be_{2t})\}$, which is not possible.
As $\pi_\ba^j(v)=v$, then $v=(v_1,\ldots,v_j,v_1,\ldots|| v_{2t+1},\ldots,v_{2t+j},v_{2t+1},\ldots)$.
Since $\wt(v)=2t$, we have $2t=(4t/2j) \wt(v_1,\ldots,v_j,v_{2t+1},\ldots,v_{2t+j})$ so $\wt(v_1,\ldots,v_j,v_{2t+1},\ldots,v_{2t+j})=j$.
Assume $j$ is an odd number, then without loss of generality we assume that the first $j$ coordinates of $v$ have an odd weight, and the first $j$ coordinates of the second half of $v$ have an even weight.
Note that $v+\pi_\ba(v)=(v_1+v_j,\ldots,v_j+v_{j-1},v_1+v_j,\ldots || v_{2t+1}+v_{2t+j},\ldots,v_{2t+j}+v_{2t+j-1},v_{2t+1}+v_{2t+j},\ldots)\in K(C)$. Moreover, since $\wt(v+\pi_\ba(v))=2t$, we have $2t=(4t/2j)\wt(v_1+v_j,\ldots,v_j+v_{j-1},v_{2t+1}+v_{2t+j},\ldots,v_{2t+j}+v_{2t+j-1})$, so $\wt(v_1+v_j,\ldots,v_j+v_{j-1},v_{2t+1}+v_{2t+j},\ldots,v_{2t+j}+v_{2t+j-1})=j$, but $\wt(v_1+v_j,\ldots,v_j+v_{j-1},v_{2t+1}+v_{2t+j},\ldots,v_{2t+j}+v_{2t+j-1})$ is even, then we have a contradiction and therefore $j$ cannot be odd.
\end{proof}

\begin{proposition}\label{tpower2-4tu2}
Let $C$ be an $\HFP(4t_\bu,2)$-code.
If $k>1$, then the length of $C$ is a power of two.
\end{proposition}
\begin{proof}
Let $4t$ be equal to $2^st'$, where $t'$ is odd. For any binary vector $x$ of even length we say that $x^{(1)},x^{(2)}$ are the projections over the first and the second half part of $x$, respectively.
Let $v$ be a vector in the kernel of $C$ different from $\be$ and $\bu$. From \Cref{aut}, as $\bb\in C$ we have that $\pi_\bb \in \Aut(K(C))$, so $\pi_\bb(v)\in K(C)$, and $v+\pi_\bb(v)\in K(C)$ since the kernel is a subspace.
We denote $v+\pi_\bb(v)$ by $v_1$. 
Since $v+\pi_\bb(v)=(v^{(1)}+v^{(2)},v^{(1)}+v^{(2)})$, we have $v_1^{(1)}=v_1^{(2)}$.
As $\ba^t \in C$, we have that $v_1+\pi_{\ba^t}(v_1)\in K(C)$.
We denote $v_1+\pi_{\ba^t}(v_1)$ by $v_2$.
Note that $v_1^{(1)(1)}=v_1^{(1)(2)}$, since $v_1+\pi_{\ba^t}(v_1)=(v^{(1)(1)}+v^{(1)(2)},v^{(1)(1)}+v^{(1)(2)},v^{(1)(1)}+v^{(1)(2)},v^{(1)(1)}+v^{(1)(2)})$.
If we repeat the same construction for $v_2$ using $\pi_{\ba^{t/2}}$, we obtain $v_3=(\alpha_3,\ldots,\alpha_3)\in K(C)$, where the length of $\alpha_3$ is $t/2$.
If we repeat this process until we can apply $\pi_{\ba^{t/2^s}}$, we obtain a vector $w$ in the kernel which can be divided in $2^s$ parts of length $t'$ which are exactly equal.
Thus, we have that $\pi_\ba^{t'}(w)=w$, and from \Cref{evenpi-4tu2}, it is only possible if $t'=1$ and $w\in \{\be,\bu\}$.

In each iteration of the construction of $v_i's$ it could happen two exceptions.
We could get a vector $v_j$ which is compounded by a combination of $\be_l's$ and $\bu_l's$. As $v_j+\pi_\ba(v_j)\in K(C)$ we have that $\wt(v_j+\pi_\ba(v_j)) \in \{0,2t,4t\}$, which is not possible by the shape of $v_j$, unless $v_j=(\be_{l},\bu_{l},\ldots,\be_{l},\bu_{l})$ with $l\in \{1,2,2t\}$ or $v_j=(\be_{l},\bu_{l},\bu_{l},\be_{l},\ldots,\bu_{l},\be_{l})$ where only appear two consecutive equal parts with $l=1$. Firstly, we suppose that $v_j=(\be_{l},\bu_{l},\ldots,\be_{l},\bu_{l})$ with $l\in \{1,2,2t\}$.
If $l=1$, then $\bomega_{4t}\in K(C)$ which contradicts \Cref{eu-4tu2}.
If $l=2t$, then $v_j=(\be_{2t},\bu_{2t})$, which contradicts \Cref{eu-4tu2}.
If $l=2$, then $v_j+\pi_\ba(v_j)=\bomega_{4t} \in K(C)$, which contradicts \Cref{eu-4tu2}.
Now, we suppose that $v_j=(\be_{l},\bu_{l},\bu_{l},\be_{l},\ldots,\bu_{l},\be_{l})$ where only appear two consecutive equal parts with $l=1$.
It is clear that there exists a permutation $\pi_{\ba^i}$ for some $i$ such that $\pi_{\ba^i}(v_j)=(\be_{l},\be_{l},\bu_{l},\bu_{l},\ldots,\be_{l},\be_{l},\bu_{l},\bu_{l})$, which is the same case as before where $l=2$.

The other exception is that we could obtain a vector $v_j\in \{\be,\bu\}$ for some $j$.
This means that $v_{j-1}=(\gamma,\ldots,\gamma)$ or $v_{j-1}=(\delta,\bar{\delta},\ldots,\delta,\bar{\delta})$ with $\gamma,\delta$ of length $i$ and $v_j=v_{j-1}+\pi_{\ba^i}(v_{j-1})\in \{\be,\bu\}$. We note that $i$ is even, on the contrary we have that  if $v_{j-1}=(\gamma,\ldots,\gamma)$ then $\pi_{\ba}^i(v_{j-1})=v_{j-1}$ which is not possible since \Cref{evenpi-4tu2}, and if $v_{j-1}=(\delta,\bar{\delta},\ldots,\delta,\bar{\delta})$ then $v_{j-1}+\pi_\ba(v_j)$ has a weight which is not in $\{0,2t,4t\}$. Thus $i$ is even. In order to try to solve this case, we have to take $v_j=v_{j-1}+\pi_{\ba^{i/2}}(v_{j-1})$. Thus, we obtain $v_j=(\gamma^{(1)}+\gamma^{(2)},\ldots,\gamma^{(1)}+\gamma^{(2)})$ or $v_j=(\delta^{(1)}+\delta^{(2)},\ldots,\delta^{(1)}+\delta^{(2)})$. If we obtain $v_j\in \{\be,\bu\}$ again, then we have to take $v_j=v_{j-1}+\pi_{\ba^{i/4}}(v_{j-1})$ and we can repeat this until we get $v_j\notin \{\be,\bu\}$ unless $v_{j-1} \in \{ \be,\bu,(\be_l,\bu_l,\ldots,\be_l,\bu_l)\}$. But, if $v_{j-1} \in \{\be,\bu\}$ then we can repeat an analogous argumentation than for $v_j$ and we obtain $v\in \{\be,\bu\}$ which contradicts the hypotheses. If $v_{j-1}=(\be_l,\bu_l,\ldots,\be_l,\bu_l)$, then we have the first exception that we have treated in the previous paragraph.

\end{proof}

\begin{theorem}\label{k1-4tu2}
Let $C$ be an $\HFP(4t_\bu,2)$-code. Then $k = 1$ for $t \geq 4$.
\end{theorem}
\begin{proof}
We suppose that $k > 1$. From \Cref{kneq2}, we have that $k\geq 3$. From \Cref{tpower2-4tu2} and \Cref{rank2t}, we have that $t$ is a power of two and $r=2t$, which contradicts \Cref{brk} for any $t \geq 3$.
\end{proof}

\subsection{\boldmath{$\HFP(2t,2,2_\bu)$}-codes} \label{2t22u}
In this subsection we assume that $C= \langle \ba, \bb, \bu : \ba^{2t}=\bb^2=\bu^2=\be \rangle$.
\begin{proposition}\label{tevensquare}
Let $C$ be an $\HFP(2t,2,2_\bu)$-code of length $4t$ with $t > 1$. Then $t$ is an even square number.
\end{proposition}
\begin{proof}
We have that $C \simeq C_{2t} \times C_2 \times \langle \bu \rangle = N \times \langle \bu \rangle$, where $N=C_{2t} \times C_2$. From \Cref{ordersquare} we have that $|C_{2t} \times C_2|$ is a square, then $t$ is a square.
The proof of \Cref{teven-4tu2} only depends on the associated permutations to the generators of the code, so we can use the same proof to get that $t$ is even.
\end{proof}

\begin{proposition} \label{boundtpow2}
Let $C$ be an $\HFP(2t,2,2_\bu)$-code of length $4t$. If $t$ is a power of two, then $t \in \{1,4\}$.
\end{proposition}
\begin{proof}
From \Cref{HDS}, as $C= C_{2t} \times C_2 \times \langle \bu \rangle$, we have an Hadamard difference set in $C_{2t}\times C_2$. From \Cref{tevensquare}, $t$ is also square, so $t$ is $2^{2s}$ for some $s$. Thus, the order and the exponent of $C_{2t}\times C_2$ are $2^{2s+2}$ and $2^{2s+1}$, respectively. From \Cref{exponentbound}, it derives that $s \leq 1$. 
\end{proof}

The next lemmas will be helpful in some proofs.
\begin{lemma}\label{eukernel}
Let $C$ be an $\HFP(2t,2,2_\bu)$-code. Then the vector $(\be_{2t},\bu_{2t})$ is not in the kernel of $C$. 
\end{lemma}
\begin{proof}
Firstly, we note that if $(\be_{2t},\bu_{2t})\in C$ then $(\be_{2t},\bu_{2t})=\ba^t$. Indeed, if $(\be_{2t},\bu_{2t}) =\ba^i\bb$ for some $i = 1,\ldots,2t$, then $(\be_{2t},\bu_{2t})^2=\bu$, which is not possible. If $(\be_{2t},\bu_{2t})=\ba^i$ for some $i=1,\ldots,2t$ then $(\be_{2t},\bu_{2t})^2=\be$ and so $i=t$.
Now we assume that $\ba^t=(\be_{2t},\bu_{2t})\in K(C)$.
Let $\ba = (\alpha_0,\alpha_1,\alpha_2,\alpha_3)$, where $\alpha_i=(a_{it+1},\ldots,a_{it+t})$.
Since $\ba \ba^t = \ba^t \ba$, we have that $\alpha_0=\alpha_1$ and $\alpha_2=\alpha_3$.
From \Cref{xi}, we have that $\ba^t=\ba+\pi_\ba(\ba)+\ldots+\pi_\ba^{t-1}(\ba)=(\be_{2t},\bu_{2t})$, so $\sum_{i=1}^{t}a_i=0$ and $\sum_{i=2t+1}^{3t}a_i=1$.
From \Cref{parseval2}, the projection of $C$ over $\Supp(\ba^t)$ have an Hadamard structure, then $\wt(\alpha_i)=t/2$, where $i=0,2$.
As $\sum_{i=2t+1}^{3t}a_i=1$, we have that $\wt(\alpha_2)$ is odd, which is not possible since $t$ is a square, and so $(\be_{2t},\bu_{2t})\notin K(C)$.
\end{proof}


\begin{lemma}\label{evenpi}
Let $C$ be an $\HFP(2t,2,2_\bu)$-code. If $v\in K(C) \setminus \langle \bu \rangle$ and $j_v$ is the smallest value such that $\pi_\ba^{j_v}(v)=v$, then $j_v$ is even.
\end{lemma}
\begin{proof}
It is an analogous proof to \Cref{evenpi-4tu2} since the associated permutations to the generators are the same.
\end{proof}

\begin{lemma}\label{1010circulant}
Let $C$ be an $\HFP(2t,2,2_\bu)$-code of length $4t$. 
Let $v$ be the codeword $\bomega_{4t}$ or $(\bomega_{2t},\bomega_{2t}\bu)$ in the kernel of $C$. If $v\neq \ba^t$, then there exists a circulant Hadamard code of length $t$. If $v=\ba^t$, then there exists an $\HFP(4t'_{\bu},2)$-code of length $4t'=t$.
The vectors $\bomega_{4t}$ and $(\bomega_{2t},\bomega_{2t}\bu)$ are not simultaneously in $K(C)$.
\end{lemma}
\begin{proof}
Let $v$ be the codeword $\bomega_{4t}$ in the kernel of $C$. We suppose that $v \in \{\ba^i,\ba^i\bb\}$ for some $i$, then $v^2=\be$ if $i$ is even, and $v^2=\bu$ if $i$ is odd. But, since the structure of the code, there are no elements whose square is the vector $\bu$. Thus, we have $v^2=\be$ implying that $i=0$ or $i=t$, hence $v\in \{\bb,\ba^t \bb,\ba^t\}$.

Firstly, we suppose that $v =\bb$. From \Cref{parseval2}, the projection of the code onto $\Supp(\bb)$ should have an Hadamard structure, which is not necessarily full propelinear. However, from $\ba^2\bb=\bb\ba^2$ we obtain $\ba^2+\pi_{\bb}(\ba^2)=\bb+\pi_{\ba}^2(\bb)=\be$ and so $\ba^2$ has the same values in the first and in the second half part. The subgroup of $C$ generated by $\ba^2$ projected over the first half part of $\Supp(\bb)$ is an $\HFP(t,2_\bu)$-code $=\hat{C}$ of length $t$. From \Cref{rifacirculant}, we know that $\hat{C}$ is a circulant Hadamard code and its dimension of the kernel is $\hat{k}=1$.

Now, we suppose that $v=\ba^t\bb$.
From \Cref{parseval2}, the projection of the code onto $\Supp(\ba^t\bb)$ should have an Hadamard structure, which is not necessarily full propelinear. However, from $\ba^2\ba^t\bb=\ba^t\bb\ba^2$ we obtain $\ba^2+\pi_{\ba^t\bb}(\ba^2)=\ba^t\bb+\pi_{\ba}^2(\ba^t\bb)=\be$ and so $\ba^2$ has the same values in the first and in the second half part, but in different order. The subgroup of $C$ generated by $\ba^2$ projected over the first half part of $\Supp(\ba^t\bb)$ is an $\HFP(t,2_\bu)$-code. As before, it is a circulant Hadamard code.

Finally, we suppose that $v=\ba^t$. Since $\ba^j\ba^t=\ba^t\ba^j$, we have that $\ba^j=(\alpha_{j,1},\alpha_{j,1},\alpha_{j,2},\alpha_{j,2})$ if $j$ is even, and $\ba^j=(\gamma_{j,1},\bar{\gamma}_{j,1},\gamma_{j,2},\bar{\gamma}_{j,2})$ if $j$ is odd. 
Let $\bb$ be $(\beta_1,\beta_2,\beta_1,\beta_2)$ for some $\beta_1,\beta_2$, since $\bb^2=\be$. As $\bb \ba^t = \ba^t \bb$, we have $\beta_1=\beta_2$, and so $\bb=(\beta_1,\beta_1,\beta_1,\beta_1)$.
From \Cref{parseval2}, the projection of the code onto $\Supp(\ba^t)$ should have an Hadamard structure, which is not necessarily full propelinear.
The subgroup $\langle(\alpha_{2,1},\alpha_{2,2})\rangle \times \langle (\beta_1,\beta_1) \rangle$ projected onto $\Supp(\ba^t)$ is an $\HFP(t_\bu,2)$-code. 

Now, let $v=(\bomega_{2t},\bomega_{2t}\bu) \in K(C)$. Note that $v$ cannot be $\bb$ neither $\ba^i\bb$ for any $i$. If $v=b$ then $\bb^2=\bu$ which is not possible. If $v=\ba^i\bb$, then $v^2=\be$ if $i$ is odd, and $v^2=\bu$ if $i$ is even, which is not possible since $t$ is even. So $v=\ba^t$, but if we apply an analogous argumentation as before we obtain an $\HFP(4t'_\bu,2)$-code.

To conclude, we suppose that $\bomega_{4t}$ and $(\bomega_{2t},\bomega_{2t}\bu)$ are in $K(C)$. Since $K(C)$ is a linear subspace, then $\bomega_{4t}+(\bomega_{2t},\bomega_{2t}\bu)=(\be_{2t},\bu_{2t})\in K(C)$ which contradicts \Cref{eukernel}.
\end{proof}

\begin{theorem}\label{tpower2}
Let $C$ be an $\HFP(2t,2,2_\bu)$-code of length $4t$.
If $k>1$, then we have some of the following:
\begin{enumerate}[i)]
    \item $k=5$, $t=4$, and $C$ is linear,
    \item $k=3$, $t=1$, and $C$ is linear,
    \item $k=3$, $t=4$, and $r=6$,
    \item $k=2$, and there exists a circulant Hadamard code of length $4t'=t$, where $t'$ is an odd square,
    \item $k=2$, and there exists an $\HFP(4t'_\bu,2)$-code of length $4t'=t$, where $t'$ is even.
\end{enumerate}
\end{theorem}
\begin{proof}
Firstly, we see that $t$ is a power of two except in two cases. We can use the same proof that in \Cref{tpower2-4tu2}, since the associated permutations to the generators are the same. To approach the two exceptions that appears in the proof of \Cref{tpower2-4tu2}, here use \Cref{evenpi} and \Cref{1010circulant} instead of \Cref{evenpi-4tu2} and \Cref{eu-4tu2}. Therefore, there exists a circulant Hadamard code of length $4t'$ or an $\HFP(4t'_\bu,2)$-code, where $4t'=t$. If we have any of these two exceptions, from \Cref{rifacirculantk1} and \Cref{k1-4tu2} we have $k=2$.

If $t$ is a power of two, then $t\in \{1,4\}$ from \Cref{boundtpow2}.  Finally, making use of the software \textsc{Magma} \cite{magma} we compute all the possibilities for $\HFP(2t,2,2_\bu)$-codes with $t\in \{1,4\}$, and we check that there is only a nonlinear code with $t=4$ whose rank is $r=6$ and dimension of the kernel $k=3$. We also obtain a linear code with $r=k=3$ when $t=1$ and a linear code with $r=k=5$ when $t=4$.

\end{proof}

\subsection{\boldmath{$\HFP(2t,4_\bu)$}-codes} \label{2t4u}
In this subsection we assume that $C=\langle \ba,\bb : \ba^{2t}=\bu^2=\be, \bb^2=\bu \rangle$.
\begin{proposition}\label{2t4u-even}
Let $C$ be an $\HFP(2t,4_\bu)$-code of length $4t$ with $t > 1$. Then $t$ is even.
\end{proposition}
\begin{proof}
From \Cref{teven-4tu2}, as the proof of $t$ is even only depends on the associated permutations to the generators of the code, we have that $t$ is even.
\end{proof}

\begin{proposition}\label{tleq8}
Let $C$ be an $\HFP(2t,4_\bu)$-code of length $4t$. If $t$ is a power of two, then $t \leq 8$.
\end{proposition}
\begin{proof}
Let $t=2^s$ for some $s$. Thus, we have $G=C_4$ is a normal subgroup of $C$ with $C/G \simeq C_{2^{s+1}}$. From \Cref{quotientbound}, the order of $G/C$ if at most $(s+5)/2$ and so $s\leq 3$.
\end{proof}

\begin{proposition}\label{eu-kernel-2t4u}
Let $C$ be an $\HFP(2t,4_\bu)$-code. If $(\be_{2t},\bu_{2t}) \in K(C)$ then $k=2$ and there exists a circulant Hadamard code of length $2t$.
\end{proposition}
\begin{proof}
Note that if $v=(\be_{2t},\bu_{2t})\in C$, then $v \in \{\ba^t,\bb,\ba^t\bb\}$. Indeed, if $v =\ba^i\bb$ for some $i = 1,\ldots,2t$, then $(\be_{2t},\bu_{2t})^2=\bu$, so $i\in \{t,2t\}$. If $v=\ba^i$ for some $i=1,\ldots,2t$ then $(\be_{2t},\bu_{2t})^2=\be$ and so $i=t$.
Now, we suppose that $v \in K(C)$. From \Cref{parseval2}, we have that the projection of $C$ over the support of $v$ have an Hadamard structure. Thus, the second half of coordinates of each codeword have weight $0$ or $t$ or $2t$. We denote the second half of coordinates of a vector $x$ by $x^{(2)}$. If we set $\pi_{\ba^{(2)}}=(1,2,\ldots,2t)$, then $\ba^{(2)}$ generate a cyclic group or order $2t$.
Firstly, we suppose that $v=\bb$, so $\bb^{(2)}=\bu$. Thus we set $\pi_{\bb^{(2)}}=I$. Hence $\langle \ba^{(2)},\bb^{(2)} \rangle \simeq C_{2t} \times C_2$, which is an $\HFP(2t,2_\bu)$-code. Since $t$ is even, we have an $\HFP(4(t/2),2_\bu)$-code, which is a circulant Hadamard code. From \Cref{rifacirculantk1}, circulant Hadamard codes have dimension of the kernel equal to $1$ so we have that $k=2$ since \Cref{parseval2}.
Now, we suppose that $v=\ba^t\bb$ then $\pi_{(\ba^t\bb)^{(2)}}=I$. We have that $\langle \ba^{(2)},(\ba^t\bb)^{(2)} \rangle \simeq C_{2t} \times C_2$ which is an $\HFP(2t,2_\bu)-$code. Since $t$ is even, we have an $\HFP(4(t/2),2_\bu)$-code. As before, it is a circulant Hadamard code, and it has dimension of the kernel equal to $1$, so $k=2$.
Finally, we suppose that $v=\ba^t$. Let $\ba^j=(\alpha_{j0},\alpha_{j1},\alpha_{j2},\alpha_{j3})$, where $\alpha_{ji}$ has length $t$. Since $\ba^j\ba^t=\ba^t\ba^j$, we have that $\alpha_{j0}=\alpha_{j1}$, and $\alpha_{j2}=\alpha_{j3}$. Since $\bb^2=\bu$ and $\bb\ba^t=\ba^t\bb$, we have that $\bb=(\beta_1,\bar{\beta}_1,\bar{\beta}_1,\beta_1)$. We have the following Hadamard matrix,
\vspace{-0.2cm}

$$ H=\left(\begin{matrix}
A_1 & A_1 & A_2 & A_2\\
B_1+A_2 & \bar{B}_1 + A_2 & \bar{B}_1+A_1 & B_1+A_1
\end{matrix}\right),
$$
where $A_1, A_2$ are the matrices whose rows are $\alpha_{j0}, \alpha_{j2}$, respectively, for $j=1,\ldots,2t$, and $B$ is the matrix whose rows are $\beta_1$. Since $H$ is an Hadamard matrix, if we set the matrix
\vspace{-0.2cm}

$$\hat{H}=\left(\begin{matrix}
A_1 & A_2\\
B_1+A_2 & \bar{B}_1+A_1 
\end{matrix}\right),
$$
then $\hat{H}$ is a matrix whose rows have weight $t$, except the last row of $(A_1,A_2)$ which is $\be$, and the distance between each two rows is $t$.
Note that $(A_1,A_2)$ is an Hadamard matrix of order $2t$. Moreover, $(\alpha_0,\alpha_2)$ generates a cyclic group of order $2t$. Thus, the rows of $(A_1,A_2)$ and their complements are an $\HFP(2t,2_\bu)$-code. As before, we have that $k=2$. 
\end{proof}

\begin{remark}
Arasu, de Launey and Ma \cite{Arasu2002} proved that a circulant complex Hadamard matrix of order $2t$ is equivalent to a $(4t,2,4t,2t)$-relative difference set in the group $C_4 \times C_{2t}$ where the forbidden subgroup is the unique subgroup of order two which is contained in the $C_4$ component. Thus, $\HFP(2t,4_\bu)$-codes are equivalent to circulant complex Hadamard matrices of order $2t$. Arasu et al. \cite{Arasu2002} also conjectured that there is no circulant complex Hadamard matrix of order greater than $16$, and they proved several non-existence results for circulant complex Hadamard matrices. The following orders up to $1000$ have yet to be excluded: $260, 340, 442, 468, 520, 580, 680, 754, 820, 884, 890$.
\end{remark}

\begin{proposition}\label{circulant-2t4u}
If there exists a circulant Hadamard code $C$ of length $4t$, then also exists an $\HFP(4t,4_\bu)$-code of length $8t$ with kernel $\langle \bu, (\be_{4t},\bu_{4t}) \rangle$.
\end{proposition}
\begin{proof}
From \cite{rifacirculant}, as $C$ is a circulant Hadamard code, we have that $C=\langle \ba, \bu\rangle$ is an $\HFP(4t,2_\bu)$-code. Let $\bb$ be $(\be_{4t},\bu_{4t})$, and $H$ be the Hadamard matrix whose rows are $\ba,\ba^2,\ldots,\ba^{4t}$. Let $\hat{H}$ be the matrix obtained by Sylvester's construction,
\vspace{-0.2cm}

$$\hat{H}=\begin{pmatrix}
H & H\\
H & \bar{H}
\end{pmatrix}
=
\begin{pmatrix}
\ba & \ba \\
\ba^2 & \ba^2\\
\vdots & \vdots\\
\ba^{4t} & \ba^{4t}\\
\ba & \ba+\bu \\
\ba^2 & \ba^2 +\bu \\
\vdots & \vdots\\
\ba^{4t} & \ba^{4t} +\bu \\
\end{pmatrix}.$$
We note that the rows of $\hat{H}$ and its complements form an $\HFP(4t,4_\bu)$-code $\hat{C}$ of length $8t$ given by $\langle (\ba,\ba), \bb\rangle$. As $(\be_{4t},\bu_{4t})+w \in \hat{C}$ for every $w \in \hat{C}$, we have that $(\be_ {4t},\bu_{4t}) \in K(\hat{C})$. As $C$ is a circulant Hadamard code, we have that $\dim(K(C))=1$ by \Cref{rifacirculantk1}. Hence, $K(\hat{C})=\langle \bu, (\be_{4t},\bu_{4t}) \rangle$.
\end{proof}

From \Cref{eu-kernel-2t4u} and \Cref{circulant-2t4u}, we have a link between the existence of $\HFP(2t,4_\bu)$-codes with $K(C)=\langle \bu, (\be_{2t},\bu_{2t}) \rangle$ and the existence of circulant Hadamard matrices. Although, even if Ryser's conjecture is true, there could exist $\HFP(2t,4_\bu)$-codes with a kernel different from $\langle \bu, (\be_{2t},\bu_{2t}) \rangle$.

In \cite{Arasu2002}, for length $2t=16$, the first row of a circulant complex Hadamard matrix is $(1,1,i,-i,i,1,1,i,-1,1,-i,-i,-i,1,-1,i)$, which corresponds to the generator $\ba=(0,0,0,1,1,0,0,0,1,1,1,0,0,1,1,1,1,0,1,1,1,1,0,1,0,1,0,\\0,0,0,1,0)$ of an $\HFP(2t,4_\bu)$-code with $r=11$ and $k=2$. We have found another $\HFP(2t,4_\bu)$-code of length $2t=16$ with $r=13$ and $k=1$ whose generator $\ba$ is $(0,0,0,0,0,0,1,0,0,1,0,1,0,1,1,1,1,0,0,0,1,1,1,1,1,1,0,1,1,0,1,\\0)$ which corresponds to the first row of a circulant complex Hadamard matrix equal to $(i,i,i,i,1,i,-i,-1,-i,i,i,-i,-1,i,-i,1)$. Now we show a conjecture about $\HFP$-codes that is equivalent to the conjecture presented by Arasu et al. \cite{Arasu2002} about circulant complex Hadamard matrices.

\begin{conjecture}\label{conjecture2t4u}
There do not exist $\HFP(2t,4_\bu)$-codes of lenght $4t$ for $t$ greater than $8$.
\end{conjecture}

\subsection{\boldmath{$\HFP(t,Q_\bu)$}-codes} \label{tQu}

In this subsection we assume that $C=\langle \bd,\ba,\bb : \bd^t=\bu^2=\be, \ba^2=\bb^2=\bu, \ba\bb\ba=\bb \rangle$ and the value of $t$ is odd. These codes are equivalent to cocyclic Hadamard matrices for which the Sylow 2-subgroup of the indexing group is $Q$, which have been considered by several authors. Baliga and Horadam \cite{hb} proved that the Williamson Hadamard matrix of order $4t$ is a cocyclic Hadamard matrix over $C_{2t} \times C_2$ with $t$ odd. \'Alvarez, Gudiel and G\"uemes \cite{agg} established some bounds on the number and distribution of 2-coboundaries over $C_t \times C_2$ to use and the way in which they have to be combined in order to obtain a $C_t \times C_2^2$-cocyclic
Hadamard matrix. Also, they completed the computational results obtained by Baliga and Horadam.
Barrera and Dietrich \cite{bd} proved that there is a 1-1 correspondence between the perfect sequences of length $t$ over $Q \cap qQ$, with $q=(1+i+j+k)/2$, and the $(4t,2,4t,2t)$-relative difference sets in $C_t \times Q$ relative to $C_2$. 

\begin{proposition}\label{proptQu}
Let $C$ be an $\HFP(t,Q_\bu)$-code. Then, up to equivalence, we have
\begin{enumerate}[i)]
\item $\pi_\bd=(1,5,\ldots,4t-3)(2,6,\ldots,4t-2)(3,7,\ldots,4t-1)(4,8,\ldots,4t)$,
\item $\pi_\ba=(1,2)(3,4)\ldots (4t-1,4t)$,
\item $\pi_\bb=(1,3)(2,4)\ldots(4t-3,4t-1)(4t-2,4t)$,
\item $\ba=(A_1,A_2,\ldots,A_t)$ where $A_i\in \{ (0,1,0,1),(1,0,1,0),(0,1,1,0),\\(1,0,0,1)\}$,
\item Knowing the value of $\bd$ is enough to define $\ba$,
\item Knowing the value of $\ba$ is enough to define $\bb$,
\item $\Pi=C_{2t}\times C_2$.
\end{enumerate}
\end{proposition}
\begin{proof}
\textit{i)},\textit{ii)} and \textit{iii)} are analogous to the proof of \Cref{genpermutations}.
\\
\textit{iv)} Let $\ba=(a_1,\ldots,a_{4t})=(A_1,\ldots,A_t)$ where
\vspace{-0.2cm}

$$A_i= (a_{4i-3},a_{4i-2},a_{4i-1},a_{4i}).$$

Since $\ba^2=\ba+\pi_\ba(\ba)=\bu$, we have that
\vspace{-0.2cm}

$$A_i=(a_{4i-3},\bar{a}_{4i-3},a_{4i-1},\bar{a}_{4i-1}).$$

\textit{v)} We have $\ba=\pi_\bd(\ba)+\bd+\pi_\ba(\bd)=\pi_\bd(\ba)+\widehat{\bd}$, where $\widehat{\bd}=\bd+\pi_\ba(\bd)$. Then $a_{4i-3}=a_{4t-3}+\sum_{j=1}^i\widehat{d}_{4j-3}$, and $a_{4i-1}=a_{4t-1}+\sum_{j=1}^i\widehat{d}_{4j-1}$, for $i=1,\ldots,t$. Since $\ba^2=\bu$, we have $a_{4i-3}=a_{4i-2}+1$ and $a_{4i-1}=a_{4i}+1$.

\textit{vi)} From $\ba\bb\ba=\bb$ we have that $\ba+\pi_\ba(\bb)=\bb+\pi_\bb(\ba^{-1})=\bb+\pi_\bb(\ba^3)=\bb+\pi_\bb(\bar{\ba})$, so $\ba+\pi_\bb(\bar{\ba})=\bb+\pi_\ba(\bb)$.
Note that $\ba+\pi_\bb(\bar{\ba})=(\hat{A}_1,\ldots,\hat{A}_t)$ where $\hat{A}_i=(1,1,1,1)$ or $(0,0,0,0)$.
\vspace{-0.2cm}

$$ \hat{A}_i=\left \{
      \begin{matrix}
         (1,1,1,1)$ if $A_i\in\{(0,1,0,1),(1,0,1,0)\} \\
         (0,0,0,0)$ if $A_i\in\{(0,1,1,0),(1,0,0,1)\}
      \end{matrix}
    \right .
$$
    
Let $\bb=(b_1,\ldots,b_{4t})=(B_1,\ldots,B_t)$ where
\vspace{-0.2cm}

$$
B_i= (b_{4i-3},b_{4i-2},b_{4i-1},b_{4i}),
$$
since $\bb^2=\bu$ we have
\vspace{-0.2cm}

$$B_i=(b_{4i-3},b_{4i-2},\bar{b}_{4i-3},\bar{b}_{4i-2}).$$

Thus,
\vspace{-0.2cm}

$$ B_i \in \left .
      \begin{matrix}
         \{(0,1,1,0),(1,0,0,1)\}$ if $A_i\in\{(0,1,0,1),(1,0,1,0)\} \\
         \{(0,0,1,1),(1,1,0,0)\}$ if $A_i\in\{(0,1,1,0),(1,0,0,1)\}
      \end{matrix}
    \right .$$
\\
\textit{vii)}
Immediately from \Cref{perm} and $t$ is odd.
\end{proof}
\begin{proposition}
Let $C$ be an $\HFP(t,Q_\bu)$-code, where $t$ is odd. Then $r=4t-1$ and $k=1$.
\end{proposition}
\begin{proof}
It is immediate from \Cref{rankkernel} and $t$ is odd.
\end{proof}

\section{Summary of results} \label{magmares}

In \Cref{table1} and \Cref{table2} we summarize the results of the current paper.
In \Cref{table1}, we show the values of the rank and the dimension of the kernel of codes that we have computed with \textsc{magma}, fulfilling the analytic results, for length $4t$. \'O Cath\'ain and R\"oder \cite{or} built the cocyclic Hadamard matrices for $t \leq 9$ that corresponds with much of the HFP-codes in \Cref{table1}, but we have computed the rank and the dimension of the kernel.
In \Cref{table2} we can see the values of the rank and the dimension of the kernel for which nonlinear $\HFP(\cdot,\cdot,\cdot)$-codes could exist for length $4t$.

\begin{table}[ht]
\centering
\begin{tabular}{|c|c|c|c|c|c|c|c|c|}
\hline
\multirow{2}{*}{t} & \multicolumn{2}{c|}{$(4t_\bu,2)$} & \multicolumn{2}{c|}{$(2t,2,2_\bu)$} & \multicolumn{2}{c|}{$(2t,4_\bu)$} &  \multicolumn{2}{c|}{$(t,Q_\bu)$} \\ \cline{2-9}
 
 & $r$ & $k$ & $r$ & $k$ & $r$ & $k$ & $r$ & $k$ \\ \hline
 
1 & 3 & 3 & 3 & 3 & 3 & 3 & x & x \\ \hline

2 & 4 & 4 & \checkmark & \checkmark & 4 & 4 & - & - \\ \hline

3 & \checkmark & \checkmark & \checkmark & \checkmark & \checkmark & \checkmark & 11 & 1 \\ \hline

\multirow{3}{*}{4} & \multirow{3}{*}{x} & \multirow{3}{*}{x} & 5 & 5 & \multirow{2}{*}{x} & \multirow{2}{*}{x} & \multirow{3}{*}{-} & \multirow{3}{*}{-} \\ \cline{4-5}
 &  & & 6 & 3 &  &  &  &  \\ \cline{4-7}
 &  &  & x & x & 7 & 2 &  &  \\ \hline

5 & \checkmark & \checkmark & \checkmark & \checkmark & \checkmark & \checkmark & 19 & 1 \\ \hline

6 & x & x & \checkmark & \checkmark & x & x & - & - \\ \hline

7 & \checkmark & \checkmark & \checkmark & \checkmark & \checkmark & \checkmark & 27 & 1 \\ \hline

\multirow{2}{*}{8} & \multirow{2}{*}{x} & \multirow{2}{*}{x} & \multirow{2}{*}{\checkmark} & \multirow{2}{*}{\checkmark} & 11 & 2 & \multirow{2}{*}{-} & \multirow{2}{*}{-} \\ \cline{6-7}
 &  &  &   &   & 13 & 1 &  & \\ \hline
 
9 & \checkmark & \checkmark & \checkmark & \checkmark & \checkmark & \checkmark & 35 & 1 \\ \hline

10 & x & x & \checkmark & \checkmark & x & x & - & - \\ \hline                 
\end{tabular}
\medskip
\caption{Rank and dimension of the kernel of Hadamard full propelinear codes with associated group $C_{2t} \times C_2$. Symbol x means that the non-existence was checked with \textsc{magma}, symbol \checkmark means that the non-existence was proved analytically, and ``-'' means that the code does not have $C_{2t}\times C_2$ as associated group.}
\label{table1}
\end{table}

\begin{table}[ht]
\centering
\begin{tabular}{|c|c|c|c|}
\hline
 $\HFP(\cdot,\cdot,\cdot)$ & $t$ & $r$ & $k$  \\ \hline
 $(4t_\bu,2)$ &  even & $\leq 2t$ & $1$  
	\\ \hline
$(2t,2,2_\bu)$ & even square & $\leq 2t$ & $1,2,3$ 
 \\ \hline
$(2t,4_\bu)$ & even 	& $\leq 2t$& $1,2,3$ 
\\ \hline    
$(t,Q_\bu)$ & odd & $4t-1$ & $1$ \\ \hline
\end{tabular}
\medskip
\caption{Values of the rank $r$, and dimension of the kernel $k$ for which nonlinear $\HFP(\cdot,\cdot,\cdot)$-codes could exist for length $4t$.}
\label{table2}
\end{table}



\end{document}